\documentclass[12pt,reqno]{amsart}

\usepackage{upgreek}     
\usepackage{enumitem}
\usepackage{graphicx}
\usepackage{mathrsfs}
\usepackage{amsthm}
\usepackage{amssymb}
\usepackage{amsmath,amsthm,amsfonts}
\usepackage{mathabx}
\usepackage{hhline}
\usepackage{textcomp}
\usepackage{amsmath,amscd}
\usepackage[all,cmtip]{xy}
\usepackage{mathtools}
\usepackage[margin=1in]{geometry}
\usepackage{lipsum}
\usepackage{extarrows}
\usepackage{upgreek}
\usepackage{bbm}
\usepackage[displaymath]{lineno}
\usepackage[singlespacing]{setspace}%
\usepackage[colorlinks=true,breaklinks=true,bookmarks=true,urlcolor=blue,
citecolor=blue,linkcolor=blue,bookmarksopen=false,draft=false]{hyperref}
\providecommand{\U}[1]{\protect\rule{.1in}{.1in}}

\theoremstyle{definition}
\newtheorem{theorem}{Theorem}[section]
\newtheorem{lemma}{Lemma}[section]

\newtheorem{defn}{Definition}[section]
\newtheorem*{theorem*}{Theorem}
\newtheorem{cor}{Corollary}[section]

\numberwithin{equation}{section}

\newcommand{\abs}[1]{\lvert#1\rvert}

\DeclareMathAlphabet{\mathpzc}{OT1}{pzc}{m}{it}

\newcommand{\ca}{\mathrm{Cat}}
\newcommand{\ftt}[1] {\mathsf{#1}}

\newcommand{\ps}{\mathrm{({PS}})_{\mathbf{G}}}
\newcommand{\pss}{\mathrm{({PS}})_c}
\newcommand{\dit}{\displaystyle{\int}}

\newcommand{\va}{\varphi}
\newcommand{\cs}{continuous }

\newcommand{\dd}{{\tt D}}
\newcommand{\dt}[1]{{\tt d}{#1}}
\newcommand{\fs}[1]{\mathbbm {#1}}
\newcommand{\eu}[1]{\EuScript {#1}}
\newcommand\Set[2]{\{\,#1\mid#2\,\}}

\newcommand*{\medcup}{\mathbin{\scalebox{1.5}{\ensuremath{\cup}}}}
\newcommand{\lc}{\mathsf{L}}

\newcommand{\h}{Hausdorff }

\newcommand{\mt}{\mathbbm {d}}

\newcommand{\zz}{\mathbb{Z}}

\newcommand{\set}[1]{\left\{#1\right\}}
\newcommand{\snorm}[2][]{\left\lVert#2\right\rVert_{#1}}

\newcommand{\zero}[1]{\mathbf{0}_{#1}}

\newcommand{\hh}{\mathcal{H}}
\newcommand{\nbd}{neighborhood \,}

\newcommand{\rr}{\mathbb{R}}
\newcommand{\nn}{\mathbb{N}}

\newcommand{\uu}{\mathcal{U}}

\newcommand{\f}{Fr\'{e}chet }

\newcommand{\bl}[1] {\mathbf {#1}}

\DeclareMathOperator{\Ind}{Ind}

\newcommand{\bb}{\mathcal{B}}
\DeclareMathAlphabet\EuScript{U}{eus}{m}{n}
\SetMathAlphabet\EuScript{bold}{U}{eus}{b}{n}
\newcommand\opn{\ensuremath{\mathrel{\mathpalette\opncls\circ}}}

\newcommand{\opncls}[2]{
	\ooalign{$#1\subseteq$\cr
		\hidewidth\raisefix{#1}\hbox{$#1{\stylefix{#1}#2}\mkern2mu$}\cr}}
\def\raisefix#1{
	\ifx#1\displaystyle
	\raise.39ex
	\else
	\ifx#1\textstyle
	\raise.39ex
	\else
	\ifx#1\scriptstyle
	\raise.275ex
	\else
	\raise.150ex
	\fi
	\fi
	\fi
}
\def\stylefix#1{
	\ifx#1\displaystyle
	\scriptstyle
	\else
	\ifx#1\textstyle
	\scriptstyle
	\else
	\ifx#1\scriptstyle
	\scriptscriptstyle
	\else
	\scriptscriptstyle
	\fi
	\fi
	\fi
}
\DeclareFontFamily{U}{mathx}{\hyphenchar\font45}
\DeclareFontShape{U}{mathx}{m}{n}{
	<5> <6> <7> <8> <9> <10>
	<10.95> <12> <14.4> <17.28> <20.74> <24.88>
	mathx10
}{}
\newcommand{\fr}{Fr\'{e}chet }

\newcommand{\Cl}[1]{\overline{#1}}

\renewcommand{\emptyset}{\varnothing}

\makeatletter
\@namedef{subjclassname@2020}{%
	\textup{2020} Mathematics Subject Classification}
\makeatother
\begin{document}

\title{Multiplicity theorems in Fr\'{e}chet setting}

\author{Kaveh Eftekharinasab}
%\thanks{The author would like to thank the reviewer for the valuable comments  }
\address{Topology lab.  Institute of Mathematics of National Academy of Sciences of Ukraine, Tereshchenkivska st. 3, Kyiv, 01601 Ukraine}

\email{kaveh@imath.kiev.ua}
\thanks {The author is grateful for the financial support within the program of support for priority research and technical (experimental) developments of the Section of Mathematics of the NAS of Ukraine for 2022-2023. Project “Innovative methods in the theory of differential equations, computational mathematics and mathematical modeling”, No. 7/1/241.  (State Registration No. 0122U000670)}


%    General info

%\date{}


\begin{abstract}
We prove multiplicity theorems for Keller $ C_c^1 $-functionals on \fr spaces and Finsler manifolds 
which are invariant under the action of a discrete subgroup. For such functionals we evaluate the minimal number of critical points by applying the Lusternik-Schnirelmann category.
\end{abstract}

\maketitle
\section{Introduction}
Functions which are invariant under a group action on manifolds or linear spaces
may have more critical points overall. This is due to the fact that
an invariant function may be viewed as a function on the quotient space which commonly has the richer topology than the ambient space. The study of critical points of invariant functions in the Banach setting has a long history and involves  theories such as Ljusternik-Schnirelman category theory, genus and cogenus
theory and various topological index theories, etc, see for example \cite{chang,chang2,gp,hb}. 

In aforementioned theories, an effectual technique in obtaining critical points of a differentiable functional is based
on deformation arguments along gradient or  pseudogradient vector fields. Deformation flows  generated by such vector fields are broadly obtained by solving a Cauchy problem. 
The other method of locating critical points is based on variational principles among which the Ekeland variational principle is the  prominent one. It states the existence of a certain minimizing sequence on a complete metric space
along which we reach the infimum value of the minimization problem. 

An attempt to extend multiplicity results to more general context of \fr manifolds was made in \cite{k},
where the Ljusternik-Schnirelman category theorem was proved for \fr Finsler manifolds. As mentioned in \cite{k}, there
are several significant obstructions to extend the deformation technique to the \fr setting:  
lack of general solvability theory of ODEs. In a \fr
space a linear ODE may have no solution for any initial value condition. Besides,
if solutions to initial value problem exist they may not be unique. In addition, 
because of the deficient topological structures of the duals of \fr spaces, cotangent bundles
do not admit smooth manifold structures and consequently the notion of pseudogradient
vector fields and Finsler structures on cotangent bundles make no sense.

In this regard, in \cite{k} it was proved a deformation result which is not implemented by
considering  pseudogradient flows for the case of \fr Finsler manifolds. Also, it was introduced the Palais-Smale condition on manifolds by using
an auxiliary function to detour the need of a smooth structure on cotangent bundles. In this paper
by applying the deformation result  we extend the Ljusternik-Schnirelmann theorem for the case differentiable functionals on \fr Finsler manifolds which are
invariant under a compact lie group action (Theorem \ref{th:clg}). For the case \fr spaces an analogous deformation result 
and a Ljusternik-Schnirelmann theorem are not available yet. Therefore, we consider another multiplicity
problem, namely evaluating the minimal number of the critical orbits of a functional which is invariant under the action of a discrete group of a \fr space (see Theorem \ref{th:co}). To prove this theorem we employ the Ekeland variational principle.

\section{
	Preliminaries }
In this section we briefly recall the basic concepts of the theory of Fr\'{e}chet spaces and establish our notations.  

By $ U \opn \mathsf{T} $ we mean that $ U $
is an open subset of a topological space $ \mathsf{T} $.  If $ \mathsf{S}$ is another topological space, then we denote by $ \eu{C} (\mathsf{T},\mathsf{S}) $ the set of  \cs mappings from  $ \mathsf{T} $ into $ \mathsf{S}$.

We denote by $\fs{F}$ a Fr\'{e}chet space whose topology is defined by a sequence of seminorms $(\snorm[\fs{F},n]{\cdot})_{n \in \nn}$, 
which we can always assume to be increasing (by considering $\max_{k \leq n}\snorm[\fs{F},n]{\cdot}$, if necessary). Moreover, the complete translation-invariant metric
\begin{equation*}\label{dm}
\mt_{\fs{F}} (x,y) \coloneqq \sum _{n \in \nn} \dfrac{1}{2^n}\cdot \dfrac{\snorm[\fs{F},n]{x-y}}{1+ \snorm[\fs{F},n]{x-y}}
\end{equation*}
induces the same topology on $\fs{F}$. We denote by
$ \fs{B}_{\mt_{\fs{F}}}(x,r) $ an open ball with center $ x $ and radius $ r>0 $. We do not write the metric $ \uprho $
when it does not cause confusion.
In what follows we consider only \f spaces over the filed $ \rr $ of real numbers. Throughout the paper we assume that   $\fs{F},\fs{E}$ are Fr\'{e}chet spaces and $\lc(\fs{E},\fs{F})$ is the set of all continuous linear mappings from $\fs{E}$ to $\fs{F}$.

A bornology $\bb_F$ on $\fs{F}$ is a  covering of $ \fs{F} $  satisfying the following axioms:
\begin{enumerate}
	\item $ \bb_F $ is stable under finite unions;
	\item if $ A \in \bb_F  $ and $ B \subseteq A$, then $ B \in \bb_F $. 
\end{enumerate}
The compact bornology on $ \fs{F} $ is the family $ \bb_F^c $ of relatively compact subsets of $ \fs{F} $ having the set of all compact subsets of $ \fs{F} $ as a base, in the sense that every $B \in \bb_F^c$ is contained in some compact set.

Let $\bb_E^c$ be the compact bornology on $ E $.
We endow the vector space $\lc(\fs{E},\fs{F})$ with the $\bb_E^c$-topology which is the topology of uniform convergence on all compact subsets of $ \fs{E} $. This is a Hausdorff locally convex topology which can be defined by the family of all seminorms obtained as follows: 
$$
\snorm[B,n]{L}   \coloneq \sup \{ \snorm[F,n]{L (e)} :  e \in B \}, 
$$
where $ B \in \bb_E^c$ and $n \in \nn $.

Let $\varphi: U \opn \fs{E} \to \fs{F}$  be a mapping. If the directional  derivatives
$$\dd \varphi(x)h = \lim_{ t \to 0} \dfrac{\varphi(x+th)-\varphi(x)}{t}$$
exist for all $x \in U$ and all $ h \in \fs{E} $, and  the induced map  $\dd \varphi(x) : U \to \lc(\fs{E},\fs{F})$ is continuous for all
$x \in U$, then  we say that $ \varphi $ is a Keller $C_c^1 $ mapping (or differentiable of class Keller $C_c^1 $) see \cite{ke}. 

 A $C_c^1$-Fr\'{e}chet manifold $ \fs{M} $ is a  manifold modeled on a Fr\'{e}chet space $ \fs{F}$ with an atlas of coordinate 
charts  such that the coordinate transition functions are all
Keller $ C_c^{1} $-mappings.

\begin{defn}\label{defni}\cite{sn1}
	Let $\fs{F}$ be a Fr\'{e}chet space, $\ftt{T}$  a topological space and $V = \ftt{T} \times \fs{F}$ the trivial bundle with fiber $\fs{F}$ over $\ftt{T}$. A Finsler
	structure for $V$ is a collection of continuous functions $\snorm[V,n]{\cdot}: V \to \mathbb{R}^+$, $n \in \nn$, such that 
	\begin{description}
		\item[(F1)] For $b \in \ftt{T}$ fixed, $ \snorm[\fs{F},n]{x}^b \coloneq \snorm[V,n]{(b,x)}$ is a collection of seminorms 
		on $\fs{F}$ which gives the topology of $\fs{F}$.
		\item [(F2)]Given $ k >1$ and $x_0 \in \ftt{T}$, there exists a neighborhood $W$ of $x_0$  such that
		\begin{equation*} \label{ine}
		\dfrac{1}{k}\snorm[\fs{F},n]{x}^{x_0} \,\leq \,\snorm[\fs{F},n]{x}^{w}\, \leq k \snorm[\fs{F},n]{x}^{x_0}
		\quad
		\text{for all} \quad  w \in W, n \in \nn, x \in \fs{F}.
		\end{equation*}
	\end{description}
\end{defn}
Suppose $\fs{M}$ is a  $C_c^1$-Fr\'{e}chet manifold modeled on $\fs{F}$. 
Let $\uppi_\fs{M} : T\fs{M} \rightarrow \fs{M}$ be the tangent bundle and let $\snorm[\fs{M},n]{\cdot}: T\fs{M} \rightarrow \mathbb{R}^+$ be a collection of continuous functions, $n \in \nn$.  We say that
$\{\snorm[\fs{M},n]{\cdot}\}_{n \in \nn} $ is a Finsler structure for 
$T\fs{M}$ if for a given $x \in \fs{M}$ there exists a bundle chart $\psi : U \times \fs{F} \simeq T\fs{M}\mid_U$ with $x \in U$  
such that
$$\{\snorm[V,n]{\cdot} \circ \, \psi^{-1}\}_{n \in \nn} $$
is a Finsler structure for $V = U \times \fs{F}$.

A Fr\'{e}chet Finsler manifold is a Fr\'{e}chet manifold together with a Finsler structure on its tangent bundle. Regular 
(in particular paracompact) manifolds admit Finsler structures.

If $\{ \snorm[\fs{M},n]{\cdot} \}_{n \in \nn}$ is a Finsler structure for $\fs{M}$ then  we can obtain a graded Finsler structure, denoted  by $( \snorm[\fs{M},n]{\cdot} )_{n \in \nn}$, that is $\snorm[\fs{M},i]{\cdot} \leq
\snorm[\fs{M},i+1]{\cdot}$ for all $i \in \nn$.

We define the length of a $C^1$-curve $\gamma : [a,b] \rightarrow M$ with respect to the $n$-th component by 
\begin{equation*}
L_n(\gamma) = \int_a^b \snorm[\fs{M},n]{\gamma'(t)}^{\gamma(t)}\dt t.
\end{equation*}
The length of a  piecewise path  with respect to the $n$-th component is the sum over the curves constituting to the path. On each connected component of $\fs{M}$, the distance is defined by
\begin{equation*} 
\uprho_n (x,y) = \inf_{\gamma} L_n(\gamma),
\end{equation*}
where infimum is taken over all  piecewise $C^1$-curve connecting $x$ to $y$. Thus,
we obtain an increasing sequence of metrics $\uprho_n(x,y)$ and define the distance $\uprho$ by
\begin{equation}\label{finmetric}
\uprho (x,y) = \displaystyle \sum_{n = 1}^{\infty} \dfrac{1}{2^n} \cdot \dfrac {\uprho_n(x,y)}{ 1+ \uprho_n(x,y)}.
\end{equation}

The distance $\uprho$ defined by~\eqref{finmetric} is a metric for $\fs{M}$ which is bounded by 1. Furthermore, the topology induced by this metric coincides with the original topology of $\fs{M}$ (see~\cite{sn1}). 

\section{Multiplicity theorems for \fr spaces}
Let $ \bl{L} $ be a topological group with the identity element $ \bl{e} $. A continuous
action of $ \bl{L} $ on a \fr space  (or manifold) $ \fs{F} $ is a mapping $ \eu{A}: \bl{L} \times \fs{F} \to \fs{F} $,
$ \eu{A}(l,m) $ written as $ l\cdot m $, such that $ \bl{e} \cdot m =m $ and 
$ (l_1 * l_2) \cdot m = l_1 \cdot (l_2 \cdot m)$ for all $ l_1,l_2 \in \bl{L} $ and all $ m \in \fs{M} $ (here $ * $ denotes the operation of $ \bl{L} $).

A set $ A \subset \fs{F}  $ is called $ \bl{L} $-invariant, if $ l \cdot m \in A $
for all $ m \in A $ and all $ l \in \bl{L} $. A functional $ \va : \fs{F} \to \rr $ is called $ \bl{L} $-invariant 
if $ \va (l \cdot m) = \va (m) $ for all $ l \in \bl{L} $ and $ m \in \fs{M} $. A map
$ \upphi : \fs{F} \to \fs{F} $ is called $ \bl{L} $-equivalent if $ \upphi (l \cdot m) = l \cdot \upphi (m) $
for all $ m \in \fs{M} $ and all $ l \in \bl{L} $. The set of fixed points of $ \fs{F} $ under the action 
is defined as
$$ \mathrm{Fix}(\bl{L})  \coloneq \Set {x \in \fs{F}}{l \cdot x =x, \, \forall l \in  \bl{L}}.$$

Let $ \bl{G} $ be a discrete subgroup of a \fr space $ \fs{F} $ and let $ q : \fs{F} \to \fs{F}/\bl{G} $ be the canonical surjection. A subset $ A \subset \fs{F} $ is called $ q $-saturated if $ A = q^{-1} \circ q (A) $.
 Suppose the space $ \fs{F}_1 $ generated by $ \bl{G} $ has the dimension $ n $. Let $ \fs{F}_2 $ be a topological complement of $ \fs{F}_1 $, so that $ \fs{F} $ is isomorphic to $ \fs{F}_1  \times \fs{F}_2, \fs{F} \cong \fs{F}_1 \times \fs{F}_2$. Let $ \bl{T}^n $ be the $ n $-torus, then $ \bl{G} \cong \zz^n $
and $ q (\fs{F}) \cong  \bl{T}^n \times \fs{F}_2$.

Let $ \va : \fs{F} \to \rr $ be a $ \bl{G} $-invariant functional of class Keller $C_c^1 $, and $ c $ a critical point of $ \va $. The set $ q^{-1}(q(c)) $ consists of the critical points of $ \va $ (since $ \va'$ is $ \bl{G} $-invariant) and is called a critical orbit of $ \va $ through $ c $.

We will extend the Palais-Smale compactness condition (\cite[Definition 3.2]{k3}) to $ \bl{G} $-invariant functionals.

\begin{defn}
	Let $ \va : \fs{F} \to \rr $ be a $ \bl{G} $-invariant functional of class Keller $C_c^1 $. 
	We say that $ \va $ satisfies $ \ps $-condition if for every sequence $ (x_n) \subset \fs{F} $
	such that $ \va(x_n) $ is bounded and $ \va'(x_n) \to 0 $, the sequence
	$ (q (x_n)) $ contains a convergent subsequence.
\end{defn}
The Lusternik-Schnirelmann category $\ca_{\mathsf{T}}A$ of a subset $A$ of a topological space $\mathsf{T}$ is the minimal number of closed sets that cover $A$
and each of which is contractible to a point in $\mathsf{T}$. If $\ca_{\mathsf{T}}A$ is not finite, we write $\ca_{\mathsf{T}}A = \infty$.

The following result concerning the Lusternik-Schnirelmann category of subsets of ANR spaces will be applied in the proof of Lemma \ref{lm:come}.
\begin{theorem}[\cite{palais}, Theorem 6.3] \label{th:anr}
	Let $ \mathsf{T} $ be an ANR space and $ A \subseteq \mathsf{T} $. Then there exists a neighborhood
	$ \uu $ of $ A $ such that $ \ca_{\mathsf{T}} \Cl{U} = \ca_{\mathsf{T}} A $.
	\end{theorem}
Let $\mathrm{Co(\fs{F})}$ be the set of compact subsets of $\fs{F}$. Define the sets
\begin{equation}
\mathcal{A}_i =  \Set{A \subset \fs{F}}{A \in \mathrm{Co(T)}, \ca_{q(\fs{F})} q(A) \geq i}, 
\end{equation}
for $i \in \nn$. In view of the property~\eqref{3} of Lemma~\ref{lem:cat} each $\mathcal{A}_i$ is a deformation invariant class of subsets of $T$.
The $i$-th Lusternik-Schnirelmann minimax value of $\varphi$ is defined by
\begin{equation}
\mu_i = \inf_{A \in \mathcal{A}_i} \sup_{x \in A}\varphi(x).
\end{equation}
It is easy to see that the sequence of numbers $\mu_i$ is increasing.

The proofs of the following two lemmas are based on the standard arguments, see for example
\cite[Lemma 3.2, Lemma 3.3]{su}. 

Let $ \mathrm{CB(\fs{F})} $ be the family of all nonempty closed and bounded subsets of $ \fs{F} $.
We define the \h metric $ \mt_{\bl{H}} $ on $ \mathrm{CB(\fs{F})} $ by
$$
\mt_{\bl{H}}(A,B) = \max \set{\sup_{a \in A} \mt_{\fs{F}}(a,B), \sup_{b \in B} \mt_{\fs{F}}(b,A)}.
$$
\begin{lemma}\label{lm:come}
	The space $ (\mathcal{A}_i, \mt_{\bl{H}}) $ is a complete metric space.
\end{lemma}
\begin{proof}
 The space $ (\mathrm{CB(\fs{F})},\mt_{\bl{H}}) $ is complete since $ \fs{F} $ is complete, cf.~\cite{mon}. 
 Thus, we only need to prove that $ \mathcal{A}_i $ is closed in  $ \mathrm{CB(\fs{F})} $. 
 Let $ (A_k) \subset \mathcal{A}_i $. Suppose $ A \in  \mathrm{CB(\fs{F})} $ and $ \mt_{\bl{H}} (A_k, A)  \to 0$. By \cite[Corollary 1.2.13]{van} and \cite [Proposition 1.2.14]{van} the space
 $ q (\fs{F}) \cong \bl{T}^n \times \fs{F}_2 \cong (\bl{S}^1)^n \times \fs{F}_2$ is an ANR.
 Therefore, by Theorem \ref{th:anr} there exists a closed neighborhood $ \uu $ of $ A $ such that
 $$
 \ca_{q(\fs{F})} (q(A)) = \ca_{q(\fs{F})} (\uu).
 $$
 As $ q^{-1}(\uu) $ is a closed neighborhood of the compact set $ A $, there exists $ k $ such that $A_k  \subset \uu$. Thereby,
 $$
 \ca_{q(\fs{F})} (q(A)) = \ca_{q(\fs{F})} (\uu) \geq \ca_{q(\fs{F})} (q(A_k)) \geq i.
 $$
 Therefore, $ A \in \mathcal{A}_i $.
\end{proof}
\begin{lemma}\label{lem:semi}
	Let $ \va : \fs{F} \to \rr $ be a $ \bl{G} $-invariant functional of class Keller $C_c^1 $.
	Then, the function 
	\begin{gather*}
		\uppsi: \mathcal{A}_i \to \rr \\
		\uppsi (A) = \max_{x \in A} \va(x)
	\end{gather*}
	is lower semicontinuous.   
\end{lemma}
\begin{proof}
	Let $ (A_k) \subset \mathcal{A}_i $.
	Suppose $ A \in  \mathcal{A}_i $ and $ \mt_{\bl{H}} (A_k, A)  \to 0$.
	For each $ x \in A $, there exists a sequence $ (x_k) \subset \fs{F} $ such that $ x_k \in A_k $
	and $ x_k \to \bl{x} $. Thus, 
	$$
	\va(\bl{x}) = \lim_{k \to \infty} \va(x_k) \leq \varliminf_{k \to \infty} \uppsi (A_k),
	$$ 
	and as $ \bl{x} \in A $ is arbitrary we have
	$$
	\uppsi (A) \leq \varliminf_{k \to \infty} \uppsi (A_k).
	$$
\end{proof}
%%%%%%%%%%%%%%%%%%%%%%%%%
We shall need the following strong version of Ekeland's variational principle.
\begin{theorem}[\cite{hir}, Theorem 4.7]\label{th:evpsf}
	Let $(\ftt{M}, \mathbbm{m})$ be a complete metric space.
	Let a functional $\upphi : \ftt{M} \rightarrow (-\infty, \infty]$ be lower semi-continuous, bounded
	from below and not identical to $\infty$. Let $ \varepsilon > 0 $ be an arbitrary real number, $ m \in \ftt{M} $
	a point such that
	$$
	\upphi(m) \leq \inf_{x \in \ftt{M}} \upphi(x) + \varepsilon.
	$$
	Then for an arbitrary $ r>0 $, there exists a point $ m_r \in \ftt{M} $ such that
	\begin{enumerate}[label=\textbf{(EK\arabic*)},ref=EK\arabic*]
		\item  \label{eq:eks1} $ \upphi(m_r) \leq \upphi (m)$;
		\item  \label{eq:eks2} $\mathbbm{m} (m_r,m) < r$;
		\item  \label{eq:eks3} $\upphi(m_r) < \upphi (x) + \dfrac{\varepsilon}{r} \mathbbm{m}(m_r,x) \, \forall x \in \ftt{M}\setminus \set{m_r}.$
	\end{enumerate}	
\end{theorem}
Also, we will apply the following properties of Lusternik-Schnirelmann category.
\begin{lemma}\cite[Proposition 2.2]{su}\label{lem:cat}
	Let $\mathsf{T}$ be a topological space, $A,B \subset \mathsf{T}$. Then
	\begin{enumerate}
		\item  If $A \subset B$, then $\ca_{\mathsf{T}}A \leq \ca_{\mathsf{T}} B$.
		\item   $\ca_{\mathsf{T}} (A \medcup B) \leq \ca_{\mathsf{T}} A + \ca_{\mathsf{T}} B$.
		\item \label{3}If $A$ is closed and $\hh: [0,t_0] \times \mathsf{T}   \to \mathsf{T}$ is a deformation, then $\ca_{\mathsf{T}}A \leq \ca_{\mathsf{T}} (\hh (t_0,A))$.
		\item If $\mathsf{T}$ is a Finsler manifold, then there exists  a neighborhood $U$ of $A$ such that $\ca_{\mathsf{T}} \overline{U} = \ca_TA$.
		\item If $\mathsf{T}$ is a connected Finsler manifold and $A$ is closed then $\ca_{\mathsf{T}} A \leq \dim A+1$, where $\dim$ is the covering dimension.
	\end{enumerate}	
\end{lemma}
\begin{theorem}\label{th:co}
Let $ \bl{G} $ be a discrete subgroup of a \fr space $ \fs{F} $. Assume that
the dimension of the space generated by $ \bl{G} $ is a finite number n.	
Let $ \va : \fs{F} \to \rr $ be a $ \bl{G} $-invariant functional of class Keller $C_c^1 $.
If $ \va $ is bounded below  and satisfies the $ \ps $-condition, then $ \va $ has $ n+1 $ critical orbits.
\end{theorem}

\begin{proof}
	Consider the increasing sequence of the Lusternik-Schnirelmann minimax values
$\mu_i = \inf_{A \in \mathcal{A}_i} \sup_{x \in A}\varphi(x)$, $1 \leq i \leq n+1.$
Define the sets
$$
\mathsf{S}_{\mu_i} \coloneq \Set{x \in \fs{F}}{\va'(x)=0 \, \&  \, \va(x) = \mu_i}.
$$
We claim that if $ \mu_i = \mu_k =\mu$  for some $ k,\, i \leq k \leq n+1 $, then $ \mathsf{S}_{\mu_i} $
contains $ k-i+1 $ critical orbits. This concludes the proof of the theorem.
We prove the claim by contradiction. Suppose that $\mathsf{S}_\mu $ contains $ m $ distinct critical orbits
$ q(x_1), \cdots, q(x_m) $ and $ m \leq k-i $. Pick the positive number $ \bl{r}  $ so that on 
$ \fs{B}(x_j,2\bl{r}), 1 \leq j \leq m, $ the canonical surjection $ q $ is injective.
Define the set
$$
\bl{B_r} \coloneq \bigcup_{j=1}^m \bigcup_{g \in \bl{G}} \fs{B}(x_j+g,r). 
$$
We show that there exists $ \varepsilon, 0 < \varepsilon^2 < \bl{r}^2 $ such that
\begin{equation}\label{eq:t1}
\snorm[B]{\va'(\bl{x})} > \varepsilon \quad  \forall B \in \bb_F^c
\end{equation}
if $ \bl{x} \in \va^{-1} ([\mu-\varepsilon^2, \mu + \varepsilon^2]) \setminus \bl{B_r} $. Because, if
\eqref{eq:t1} is not valid, then there exists a sequence $ (x_{\bl{j}}) \subset \fs{F} \setminus \bl{B_r} $
such that 
$$
\abs{\va(x_{\bl{j}})} \leq \mu + \dfrac{1}{\bl{j}} \, \text{and}\, \snorm[B,n]{\va'(\bl{x})} \leq \dfrac{1}{\bl{j}}
\quad \forall n \in \nn \, \text{and} \, \forall B \in \bb_F^c.
$$
Since $ \va $ satisfies the $ \ps $-condition we may assume that $ q(x_{\bl{j}}) \to q(\bar{x}) $ 
for some $ \bar{x} \in \fs{F} $. As $ \va $ and $ \va' $ are $ \bl{G} $-invariant, we may suppose that 
$ x_{\bl{j}} \in [0,1]^n \times \fs{F}_2 $. Whence, $ x_{\bl{j}} \to \bar{x} $ yields $ \bar{x} \in \fs{F} \setminus \bl{B_r} $ and $ \va(\bar{x}) = \mu $ and $ \va'(\bar{x}) =0 $ which is impossible because $ \bl{B_r} $
is a neighborhood of $ \mathsf{S}_{\mu} $. There exists $ A \in \mathcal{A}_k $ such that
$$
\uppsi(A) = \max_A \va \leq \mu + \varepsilon^2.
$$
This is achievable by the definition of $ \mu_k $. Let $ \bl{A} = A\setminus \bl{B}_{2\bl{r}} $. In virtue of Lemma \ref{lem:cat} we obtain
\begin{align*}
	k &\leq \ca_{q(\fs{F})} q(A) \\
	  &\leq \ca_{q(\fs{F})} \big( q(\bl{A}) \medcup q(\bl{B}_{2\bl{r}}) \big) \\
	  &\leq \ca_{q(\fs{F})} q(\bl{A}) +  \ca_{q(\fs{F})}q(\bl{B}_{2\bl{r}})\\
	  &\leq \ca_{q(\fs{F})} q(\bl{A}) +  m \quad \text{by the definition of the category}\\
	  &\leq \ca_{q(\fs{F})} q(\bl{A}) + k- i.
\end{align*}
Thus, $ \bl{A}  \in \mathcal{A}_i$.

By Lemma \ref{lm:come}, the space $ (\mathcal{A}_i, \mt_{\bl{H}}) $ is complete. Also, by Lemma
\ref{lem:semi} the function $ \uppsi: \mathcal{A}_i \to \rr $ is lower semicontinuous. So, we can employ
the Ekeland variational theorem \ref{th:evpsf}. By the latter theorem there exists $ \eu{A} \in \mathcal{A}_i $ such that
\begin{enumerate}[label=\textbf{(P\arabic*)},ref=P\arabic*]
	\item $ \uppsi (\eu{A}) \leq \uppsi (\bl{A}) \leq \uppsi(A) \leq \mu + \varepsilon^2$,
	\item $\mt_{\bl{H}}(\eu{A}, \bl{A}) \leq \varepsilon$,
	\item  \label{eq:t2} $\uppsi (S) > \uppsi (\eu{A}) - \epsilon \mt_{\bl{H}} (\eu{A}, S)$ \quad $ S \in \mathcal{A}_i, S \neq \eu{A}. $
\end{enumerate}
As $ \bl{A} \cap \bl{B}_{2 \bl{r}} = \emptyset $ and $ \mt_{\bl{H}}(\eu{A}, \bl{A}) \leq \varepsilon  \leq \bl{r}$,
then $ \eu{A} \cap \bl{B}_{2\bl{r}} = \emptyset $.
Also, the set 
$$
C \coloneq \Set{s \in \eu{A}}{\mu - \varepsilon^2 \leq \va(s)}
$$
is a subset of  $\va^{-1} ([\mu-\varepsilon^2, \mu + \varepsilon^2]) \setminus \bl{B_r}$. The set
$ C $ is closed and as $ \va $ is continuous, then it is compact. By \eqref{eq:t1} for each
$y \in C$ there exists $ h_{B,y} \in B $ such that 
\begin{equation}\label{eq:t3}
\langle \va'(y), h_{B,y}\rangle < - \varepsilon.
\end{equation}
Since $ \va' $ is continuous, it follows from \eqref{eq:t3} that there exists $ r_y>0 $ such that for all 
$ g \in \bl{G} $ and all $ h \in \fs{F} $ with $ \snorm[\fs{F},n]{h} < r_y $ we obtain
$$
\langle \va'(y+g+h), h_{B,y}\rangle < - \varepsilon.
$$
Since $ C $ is compact, we can find a subcovering $ C_1, \cdots, C_{\bl{n}} $ defined by
$$
C_{\bl{i}} = \fs{B}(y_{\bl{i}}, r_{y_{\bl{i}}}) \quad 1 \leq i \leq \bl{n}.
$$
Define the functions $ \upphi_{\bl{i}} : \fs{F} \to [0,1] $ by
\begin{gather*}
\upphi_\bl{i}(x) =\begin{cases} 
\dfrac {\sum_{g \in \bl{G}} \mt_{\fs{F}} (x+g, \complement C_{\bl{i}} )}
{\sum_{k=1}^{\bl{n}}\sum_{g \in \bl{G}} \mt_{\fs{F}} (x+g, \complement C_{k} )} &   x \in \bigcup_{i=1}^\bl{n} C_j, \\
0 & \text{otherwise}. 
\end{cases}
\end{gather*}
Fix a $ \bl{G} $-invariant continuous function  $\upphi : \fs{F} \to [0,1]$ such that
\begin{gather*}
\upphi(x) =\begin{cases} 
1 &   \mu \leq \va (x), \\
0 & \va (x) \leq \mu - \varepsilon^2 . 
\end{cases}
\end{gather*}
Let $ r_{min} = \min_{1 \leq \bl{i} \leq \bl{n}}{r_{y_{\bl{i}}}}$. Define the curve 
$ \eu{I} \in \eu{C }([0,1] \times \fs{F}, \fs{F})  $ by
$$
\eu{I}(t,x) = x + t r_{min}\upphi(x) \sum_{\bl{i}=1}^{\bl{n}} \uppsi_{\bl{i}}(x) (h_{B,y_{\bl{i}}}).
$$
For all $ x \in \fs{F} $, all $ g \in \bl{G} $ and all $ t \in [0,1] $ we have
$$
\eu{I}(t,x+g) = \eu{I}(t,x) + g.
$$
Lemma \ref{lem:cat} implies that
$$
\ca_{q(\fs{F})} \big( q(\eu{I}(1,\eu{A}))\big) \geq   \ca_{q(\fs{F})} \big( q(\eu{A})\big) \geq i,
$$
whence, as $ \eu{I}(1,\eu{A}) $ is compact, $ \eu{I}(1,\eu{A}) \in \eu{A}_i $.
By the mean value theorem (see \cite{ke}) and \eqref{eq:t2} for each $ y \in C $,
there is $ \bl{t} \in (0,1) $ such that
\begin{align}
\va (\eu{I}(1,\eu{A}))- \va (x) &=
\langle \va'( \eu{I}(\bl{t},\eu{A})) , r_{min}\upphi(x) \sum_{\bl{i}=1}^{\bl{n}} \uppsi_{\bl{i}}(x) (h_{B,y_{\bl{i}}})    \rangle \nonumber \\
&= r_{min}\upphi(x) \sum_{\bl{i}=1}^{\bl{n}} \uppsi_{\bl{i}}(x) \langle \va'(x + \bl{t}r_{min}\upphi(x) \sum_{\bl{i}=1}^{\bl{n}} \uppsi_{\bl{i}}(x) (h_{B,y_{\bl{i}}}) ) ,
 h_{B,y_{\bl{i}}}  \rangle \nonumber \\
&\leq -\varepsilon r_{min} \upphi (x) \label{eq:t5}.
\end{align} 
If $ x \in C $, then $ \upphi(x) = 0$ and $ \va (\eu{I}(1,X)) = \va(x)$.
Let $ \bl{y} \in \eu{A} $ so that $ \va (\eu{I}(1,\bl{Y})) = \uppsi(C) $.
Then, 
$$
\mu \leq \va (\eu{I}(1,\bl{y})) \leq \va (\bl{y}).
$$
Thus, $ \bl{y} \in C $ and $ \upphi(\bl{y}) =1 $. It follows from
\eqref{eq:t5} that
$$
\va(\eu{I}(1,\bl{y})) - \va(\bl{y}) \leq -\varepsilon r_{min}.
$$
Therefore,
$$
\uppsi (S) + \varepsilon r_{min} \leq \va (\bl{y}) \leq \uppsi (\eu{A}).
$$
However, $ \mt_{\bl{H}} (\eu{A},S) \leq r_{min} $ by the definition of $ S $.
Hence,
$$
\uppsi (S) + \varepsilon \mt_{\bl{H}}(\eu{A},S) \leq \uppsi (\eu{A})
$$
which contradicts \eqref{eq:t2} and concludes the proof. 
 \end{proof}
%%%%%%%%%%%%%%%%%%%%
\section{Multiplicity theorems for \fr manifolds}
Henceforth we assume that $\fs{M}$ is a connected $C^1$-Fr\'{e}chet manifold modeled on $ \fs{F} $ endowed with a complete Finsler  metric $ \uprho $ \eqref{finmetric}, and  that $\varphi : \fs{M} \to \rr$ is a non-constant Keller $C_c^1$-functional. 
Let $x\in \fs{M}$, we shall say that $x$ is a critical point of $\varphi$ if $(\varphi \psi^{-1})'(\psi(x)) =0$ for
a chart $(x\in U,\psi)$ and hence for every chart whose domain contains $x$.

Let $( \snorm[\fs{M},n]{\cdot} )_{n \in \nn}$ be a graded Finsler structure on $T\fs{M}$. 
Define a closed unit semi-ball centered at the zero vector $\zero{x}$ of $T_x\fs{M}$ by $$\fs{B}^n(\zero{x},1) = \{ h \in T_x\fs{M} : \, \snorm[\fs{M},n]{h}^x \leq 1\}$$ for each $x \in  \fs{M}$ and each 
$\snorm[\fs{M},n]{h}^x $.
Let $$\fs{B}_{\infty}(\zero{x}) = \bigcap_{n=1}^{\infty} \fs{B}^n(\zero{x},1).$$
The set $\fs{B}_{\infty}(0_x)$ is not empty and  infinite because it can be identified with a convex neighborhood of the zero of the
Fr\'{e}chet space $U \times  \fs{F}$, where $U$ is an open neighborhood of $x$.

Let $\varphi: \fs{M} \to \rr$ be a $C^1$-functional and $x \in \fs{M}$. Define
\begin{equation}
\Phi_{\varphi}(x) = \inf \big \{ \varphi'(x,h) : h \in \fs{B}_{\infty}(\zero{x}) \big\}.
\end{equation}
\begin{defn}[The $ \pss $-condition for \fr manifolds, \cite{k}]\label{def:ps}
	We say that a $C^1$- functional $\varphi: \fs{M} \to \rr$ satisfies the Palais-Smale condition at a level $c \in 
	\rr$, $\pss$ in short, in a set $A \subset \fs{M}$ if any sequence $( m_i)_{i \in \nn} \subset A$ such that $$\varphi(m_i) \to c  \quad \mathrm{and} \quad
	\Phi_{\varphi}(m_i)  \to 0, $$ has a convergent subsequent.
\end{defn}
 We denote by
$\mathrm{Cr(\varphi)}$ the set of critical points of $\varphi$, and for $c \in \rr$
$$\mathrm{Cr}(\varphi,c) = \{ x\in \mathrm{Cr(\varphi)},\varphi(x)=c\},$$
$$\varphi^c = \{ x \in \fs{M} : \varphi(x) \leq c \}.$$
A mapping $\hh \in \eu{C}([0,1] \times \fs{M} \to \fs{M})$ is called a deformation if $\hh (0,x) =x$ for all $x \in M$.
Let $C$ be a subset of $\fs{M}$, we say that  $\hh$ is a $C$-invariant for an interval $I \subset [0,1]$ 
if $\hh(t,x) = x$ for all $x\in C$ and all $t \in I$.

A family $\mathcal{F}$ of subset of $\fs{M}$ is said to be deformation invariant if for each $A \in \mathcal{F}$ and each deformation $\hh$
for $\fs{M}$, $\hh_1(x) \coloneq \hh(1,x)$, it follows that $$\hh_1 (A) \in \mathcal{F}.$$

A compact Lie group $ \bl{L} $ admits a unique $ \bl{L} $-invariant measure $ \mu = \mu_{\bl{L}} $, called the Haar measure, such that $ \mu (\bl{L})=1 $.

A crucial step in extending  the Ljusternik–Schnirelmann theory  is the determination of topological indices, namely:
\begin{defn}
	Let $ \mathsf{D} \coloneq \Set{A \subset \fs{M}}{ A \text{ is closed and}\, \bl{L}\text{-invariant}} $.
	An $ \bl{L} $-index on $ \fs{M} $ is a map
	 $ \Ind : \mathsf{D} \to \nn \medcup \set{\infty}$ such that for all $ A,B \in \mathsf{D} $ we have
	 \begin{enumerate}[label=\textbf{(I.\arabic*)},ref=I.\arabic*]
	 \item \label{def:i1} $\Ind (A) = 0$ if and only if $ A = \emptyset $; 	
	 \item \label{def:i2} $ \Ind(A \cup B) \leq \Ind (A) + \Ind (B) $;
	 \item  \label{def:i3} if $\upphi :  A \to B$ is $ \bl{L} $-equivalent, then $ \Ind{A} \leq \Ind (B) $;
	 \item  \label{def:i4} if $ A $ is compact, then there exists an $ \bl{L} $-invariant, open \nbd $ \uu $ of $ A $
	 such that $ \Ind(\Cl{\uu})  = \Ind (A) $. 
 	\end{enumerate}	
\end{defn}
\begin{lemma}\label{lem:inv}
	Let $ \va : \fs{M} \to \rr $ be an $ \bl{L} $-invariant functional of class Keller $ C_c^1 $.
	Let $ (\uu, \psi) $ be an arbitrary chart and $ \va_{\psi} $ the local representative of
	$ \va $ in the chart. Then
	$$
	\langle \va_{\psi}'  (\psi (l \cdot x)), \psi (y) \rangle 
	= \langle \va'(\psi(x), \psi (l^{-1}\cdot y)) \rangle \quad \forall x,y \in \uu; \forall l \in \bl{L}.
	$$
	
\end{lemma}
\begin{proof}
	\begin{align*}
		\langle \va_{\psi}'  (\psi (l \cdot x)), \psi (y) \rangle &= \lim_{ t \to 0} \dfrac{1}{t} \bigg(
		\va_{\psi} \big( l \cdot \big(\psi^{-1}(\psi(x)+ t \psi(l^{-1} \cdot y)) \big)  \big) - \va_{\psi} (\psi (l \cdot x)) 
		\bigg) \\
		&= \lim_{ t \to 0} \dfrac{1}{t} \bigg(
		\va_{\psi} \big( \psi^{-1}(\psi(x)+ t \psi(l^{-1} \cdot y))   \big) - \va_{\psi} (\psi (x)) 
		\bigg) \\
		&= \langle \va'(\psi(x), \psi (l^{-1}\cdot y)) \rangle.
	\end{align*}
\end{proof}

Now for $ \bl{L} $-invariant functionals, we will extend the deformation results \cite[Lemma 3.1, Corollary 3.5]{k}to make the deformation $ \bl{L} $-equivalent. 

\begin{theorem}\label{th:eqdef}
 Assume $\varphi: \fs{M} \to \rr$ is $ \bl{L} $-invariant. Let $B$ and $A$ be closed disjoint, $ \bl{L} $-invariant subsets of
	$\fs{M}$ and let $A$ be compact. Suppose $k>1$ and $\epsilon>0$ are
	such that $\Phi_\varphi(x) < -2\epsilon(1+k^2)$ for all $x \in A$. Then there exist $t_0>0$ and
	$\bl{L}$-equivalent deformation $\hh$ for  $ [0,t_0)$ such that
	\begin{enumerate} 
		\item $ \hh (t,x) = x, \quad \forall t \in B, \, \forall t \in [0,t_0) $,
		\item  $\rho (\mathcal{H} (t,x),x) \leq kt, \quad \forall x \in M$,
		\item  $
		\varphi(\hh(t,x)) -\varphi (x) \leq - 2\epsilon(1+k^2) t, \quad \forall x \in M.
		$\label{defo}
	\end{enumerate}
\end{theorem}
\begin{cor}\label{cor:eqdef}
Let  $\varphi: M \to \rr$ be a Keller $C_c^1$ closed non-constant function. Suppose $\varphi$ is $ \bl{L} $-invariant and satisfies the Palais-Smale condition at all levels. 
\begin{enumerate}
	\item  If for $c \in \rr$ and $\delta > 0$ we have $$\varphi^{-1}[c-\delta, c+\delta] \cap \mathrm{Cr(\varphi)} = \emptyset,$$ 
	then there exists $t_1< t_0$ and $0< \epsilon < \delta$ such that
	\begin{equation}
	\hh (t_1, \varphi^{c+\epsilon}) \subset \varphi^{c-\epsilon}.
	\end{equation}
	\item \label{probo} If $\varphi$ has finitely many critical points, and for $c \in \rr$ if $U$ is an $ \bl{L} $-invariant, open neighborhood of $\mathrm{Cr}(\varphi,c)$ ($U =\emptyset$ if $\mathrm{Cr}(\varphi,c) = \emptyset$), then there exist $t_1 < t_0$ and $\epsilon > 0$ such
	that
	\begin{equation}
	\hh (t_1, \varphi^{c+\epsilon}\setminus U) \subset \varphi^{c-\epsilon}.
	\end{equation}
\end{enumerate}	
\end{cor}
The proofs of Theorem \ref{th:eqdef} and Corollary \ref{cor:eqdef} are just modifications of the proofs of
 Lemma 3.1 and Corollary 3.5 in \cite{k}. We just need to show that if $ \hh $ is a deformation obtained in
 \cite[Lemma 3.1]{k}, then $ \hh_l(t,m) = \dit_{\bl{L}} l^{-1} \cdot \hh (t,l\cdot m) \dt{\mu}$ is $ \bl{L} $-equivalent. To see the latter, let $ g \in \bl{L}$. Then, 
\begin{align*}
 \hh_l(t,g \cdot m) &= \dit_{\bl{L}} l^{-1} \cdot \hh (t, g\cdot (l\cdot m)) \dt{\mu} \\
 &= g \cdot\dit_{\bl{L}}  (l * g)^{-1} \cdot \hh (t, ((l* g) \cdot  m) \dt{\mu} \\
 &= g \cdot  \hh_l(t, m).
\end{align*}

Now we extend the Ljusternik–Schnirelmann theorem \cite[Theorem 3.10]{k}. 

Given $ i \in \nn $, we define the sets
$$
\mathcal{K}_i \coloneq \Set{A \in \fs{M}}{A \, \text{is compact and} \, \bl{L}\text{-invariant with} \Ind(A) \geq i},
$$
and the values $ c_i $ by
$$
c_i \coloneq \inf_{A \in \mathcal{K}_i}\max_{x \in A}.
$$
\begin{theorem}\label{th:clg}
	Let $ \bl{L} $ be a compact lie group that acts on a $ C_c^1 $-\fr Finsler manifold $ M $.
	Let $ \va : \fs{M} \to \rr$ be a non-constant closed, $ \bl{L} $-invariant functional of class
	Keller $ C_c^1 $. Given any $ n \in \nn $ such that $ n \leq m $ for some $ m \in \nn $ such that
	$c \coloneq c_n = c_k > -\infty $. If $ \va $ satisfies the $ \pss $-condition, then 
	$ \Ind (\mathrm{Cr}(\varphi,c)) \geq m- n+1 $.
\end{theorem} 
 
 \begin{proof}
 	From Definition \ref{def:ps} it follows directly that  the set $ \mathrm{Cr}(\varphi,c)) $ is compact and it is
 	$ \bl{L} $-invariant by Lemma \ref{lem:inv}.
 	In virtue of \eqref{def:i4} we may find an $ \bl{L} $-invariant, open \nbd $ \uu $ of
 	$ \mathrm{Cr}(\varphi,c)) $ such that $ \Ind(\Cl{U}) = \Ind (\mathrm{Cr}(\varphi,c))) $.
 	By the definitions of the values $ c_i $, we may find $ A \in \mathcal{K}_i $ such that
 	\begin{equation*}
 	\max_A \va \leq c + \epsilon.
 	\end{equation*}
 	Let $ \eu{A} = A \setminus \uu $. Then by \eqref{def:i2} and \eqref{def:i3} we will have
 	\begin{align} \label{eq:tg2}
 		i &\leq \Ind (A) \nonumber \\
 		  &\leq \Ind (\eu{A}) + \Ind(A \cap \Cl{\uu}) \nonumber \\
 		  &\leq \Ind (\eu{A}) + \Ind( \Cl{\uu}) \nonumber \\
 		  &= \Ind (\eu{A}) + \Ind (\mathrm{Cr}(\varphi,c))
 	\end{align}
 	Observe that $ \eu{A} \subseteq \va^{c+ \epsilon} \setminus \uu $. Thus, by Theorem there is an
 	$ \bl{L} $-equivalent deformation $ \hh $ such that by applying Corollary \ref{cor:eqdef} for some
 	$ 0 < t_0 <1 $ we will get
 	\begin{equation*}
 	\bl{A} \coloneq \hh (t_0, \eu{A}) \subseteq \va^{c-\epsilon}.
 	\end{equation*}
 	Since $ \eu{A} $ is compact and $ \bl{L} $-invariant, and $ \hh(t_0,\cdot) $ is $ \bl{L} $-equivalent,
 	it follows that $ \bl{A} $ is compact and $ \bl{L} $-invariant and 
 	\begin{equation*}
 	\max_A \va \leq c - \epsilon.
 	\end{equation*}
 	Moreover,  $  \Ind( \bl{A}) \leq n-1 $ by the definition of $ c_n=c $.
 	Thus, by \eqref{def:i3} 
 	\begin{equation}\label{eq:tg1}
 	\Ind (\eu{A}) \leq \Ind( \bl{A}) \leq n-1.
 	\end{equation}
 	It follows from \eqref{eq:tg1} and \eqref{eq:tg2} that 	$ \Ind (\mathrm{Cr}(\varphi,c)) \geq m- n+1 $.
 	Furthermore, $  {\mathrm{Cr}(\varphi,c)} \neq \emptyset$ by \eqref{def:i1}.
 \end{proof}
\bibliographystyle{amsplain}

\end{document}